\documentclass[11pt]{article}

\usepackage[T1]{fontenc}
\usepackage{lmodern}
\usepackage{amsmath,amssymb,amsthm,mathtools}
\usepackage{enumitem}
\usepackage{geometry}
\usepackage{microtype}
\usepackage[hidelinks]{hyperref}
\usepackage{indentfirst}

\allowdisplaybreaks

\newtheorem{theorem}{Theorem}[section]
\newtheorem{lemma}[theorem]{Lemma}
\newtheorem{claim}{Claim}
\newtheorem{conjecture}[theorem]{Conjecture}

\newcommand{\ceil}[1]{\left\lceil #1\right\rceil}
\newcommand{\floor}[1]{\left\lfloor #1\right\rfloor}
\newcommand{\Nb}{N^{\mathrm b}}
\newcommand{\Nr}{N^{\mathrm r}}

\title{On 
monochromatic path covers conjecture of Erd\H{o}s--Gy\'arf\'as}
\author{Hangdi Chen\thanks{Email: \texttt{chenhangdi188@126.com}}\\
\small Fujian Key Laboratory of Financial Information Processing, Putian University,\\[-2pt]
\small Putian 351100, P.R. China
\and
Yaojun Chen\thanks{Email: \texttt{yaojunc@nju.edu.cn}}\\
\small School of Mathematics, Nanjing University, Nanjing 210093, P.R. China}
\date{}

\begin{document}
\maketitle

\begin{abstract}
Erd\H{o}s and Gy\'arf\'as conjectured in 1995 that, in every red--blue
edge-coloring of a complete graph $K_n$, the vertex set can
be covered by at most $\sqrt n$  monochromatic paths, all of the same color.
Pokrovskiy, Versteegen and Williams (JCT-B, 2026) proved the conjecture for all
sufficiently large $n$.  In this paper, by using minimal counterexample method, we confirm the conjecture completely.
\end{abstract}

\medskip

{\bf Keywords:} Complete graph, edge-coloring, monochromatic path covers

\section{Introduction}

A monochromatic path is one whose edges are colored the same.
Gerencs\'er and Gy\'arf\'as~\cite{GerencserGyarfas1967} started to investigate the monochromatic path covers of a 2-edge-colored complete graph $K_n$, and established the following result.

\begin{theorem}[Gerencs\'er and Gy\'arf\'as~\cite{GerencserGyarfas1967}]
\label{thm:two-paths}
The vertex set of every red--blue edge-colored $K_n$ can be covered
by two monochromatic paths.
\end{theorem}

The two paths in Theorem~\ref{thm:two-paths} need not have the same color.
%As Gy\'arf\'as recalls in his survey~\cite{Gyarfas2016}, Erd\H{o}s initially
%interpreted the theorem as 
By imposing a same-color requirement, Erd\H{o}s and Gy\'arf\'as  studied the minimum number of
monochromatic paths of one color needed to cover all vertices of a
red-blue edge-colored complete graph, and  obtained the following.
%bound~\cite{ErdosGyarfas1995}.

\begin{theorem}[Erd\H{o}s and Gy\'arf\'as~\cite{ErdosGyarfas1995}]
\label{thm:two-root-n}
The vertex set of every red--blue edge-colored $K_n$
can be covered by at most $2\sqrt n$ monochromatic paths, all of the same
color.
\end{theorem}

They further conjectured that the factor $2$ can be removed.

\begin{conjecture}[Erd\H{o}s and Gy\'arf\'as~\cite{ErdosGyarfas1995}]
\label{conj:EG}
The vertex set of every red--blue edge-colored $K_n$
can be covered by at most $\sqrt n$ monochromatic paths, all of the same
color.
\end{conjecture}

The paths appearing in Theorems \ref{thm:two-paths} and \ref{thm:two-root-n} and in Conjecture \ref{conj:EG} are allowed to intersect. Recently, Pokrovskiy, Versteegen and Williams \cite{PokrovskiyVersteegenWilliams2026} proved Conjecture \ref{conj:EG} for all sufficiently large $n$.

\begin{theorem}[Pokrovskiy, Versteegen and Williams~\cite{PokrovskiyVersteegenWilliams2026}]
\label{thm:PVW}
For every $n>20^{40}$, the vertex set of every red--blue edge-colored
$K_n$  can be covered by at most $\sqrt n$
monochromatic paths, all of the same color.
\end{theorem}

In this paper,  we confirm Conjecture \ref{conj:EG} completely.
%Our main result removes the lower bound on $n$.

\begin{theorem}
\label{thm:main}
For every positive integer $n$, the vertex set of every red--blue
edge-colored  $K_n$ can be covered by at most
$\sqrt n$ monochromatic paths, all of the same color.
\end{theorem}

%\paragraph{Sketch of the proof.}
%Suppose that a minimal counterexample exists, and write
%$n=c^{2}+r$, where $c=\lfloor\sqrt n\rfloor$.
%Minimality implies that any red path and any blue path have at most
%$r$ vertices in common. Let $Q$ be a longest monochromatic path,
%assumed to be blue, and set $X=V(Q)$ and $Y=V(G)\setminus X$.
%Partition $Y$ into the set $Y_0$ of vertices having no blue neighbour
%in $X$ and its complement $Y_1$. The maximality of $Q$ forces the
%predecessors of the blue neighbours of each vertex in $Y_1$, together
%with an endpoint of $Q$, to form a red clique in $X$. Alternating
%red-path extensions and efficient red covers of the unused vertices
%of $X$ then strongly restrict the blue neighbourhoods of vertices
%in $Y_1$.

%For $c\ge4$, let $\mu$ be the maximum blue degree from $Y_1$ into $X$.
%If $2\mu\le r$, large common red neighbourhoods yield either a red
%path meeting $Q$ in more than $r$ vertices or a red path cover of size
%at most $c$. If $2\mu\ge r+1$, suitable triples of vertices outside
%$Q$ are encoded in a $3$-uniform hypergraph. Matching and gap
%arguments along $Q$ again produce one of the same two contradictions.
%Hence $c\le3$. The cases
%$c=1,2,3$ are handled directly using bipartite path-cover estimates
%and short alternating red paths, completing the contradiction.
Related problems of this type have also been studied, see \cite{Allen2008,BessyThomasse2010,EugsterMousset2018,Gyarfas1989,GyarfasLehel1973,Gyarfas2016,LuczakRodlSzemeredi1998}, for instance. At the end of this section, we introduce some notation that will be used throughout this paper. 
For a graph $G$, let $V(G)$ be the \emph{vertex set} of $G$ and  $E(G)$ be the \emph{edge set} of $G$. For a subset $A$ of $V(G)$, let $G[A]$ denote the subgraph of $G$ induced by $A$.  Let $N_G(v)$ denote the neighbourhood of $v$ consisting of all vertices adjacent to $v$. The \emph{degree} of a vertex $v$, denoted by $d_G(v)$, is the size of $N_G(v)$.  When no confusion can occur, we will omit the subscript $G$. 

A \emph{red--blue edge-coloring} of $G$ assigns either red or
blue to every edge of $G$.  A path or cycle is \emph{red} (respectively,
\emph{blue}) if all its edges are red (respectively, blue).  The \emph{length} of a
path is the number of edges.  We permit paths of length zero, each consisting
of a single vertex; such a path may be regarded as either red or blue.  Thus, a single-vertex path may be assigned either color when considered as monochromatic paths. A collection $\mathcal P$ of red paths is a \emph{red path cover} of $G$ if every
vertex of $G$ lies on at least one member of $\mathcal P$.  A blue path cover is defined analogously. It is worth noting that paths in a cover need not be vertex-disjoint.

\section{Preliminaries}

We first recall two classical results concerning monochromatic paths. Following
Erd\H{o}s and Gy\'arf\'as~\cite{ErdosGyarfas1995}, we say that a red--blue
edge-coloring of $K_n$ is a \emph{cut coloring} if there exists a longest
monochromatic path whose endpoints are joined by an edge of the same color as
the path.

\begin{lemma}[Erd\H{o}s and Gy\'arf\'as~\cite{ErdosGyarfas1995}]
\label{lem:EG-extension}
Suppose that a red--blue edge-coloring of $K_n$ is not a cut coloring. Let
$A$ be the vertex set of a longest monochromatic path. By symmetry, we may
assume that this path is blue. Let $B\subseteq V(K_n)\setminus A$. If $|B|<\ceil{|A|/2}$, then there exists a red path containing every vertex of
$B$ together with $|B|+2$ vertices of $A$.
\end{lemma}

\begin{lemma}[Gerencs\'er and Gy\'arf\'as~\cite{GerencserGyarfas1967}]
\label{lem:long-path}
Every red--blue edge-coloring of $K_n$ contains a monochromatic path on at
least $\floor{2n/3}+1$ vertices.
\end{lemma}

We shall also use the following three auxiliary lemmas due to Pokrovskiy,
Versteegen, and Williams~\cite{PokrovskiyVersteegenWilliams2026}. Recall that a
graph $G$ is bipartite with parts $X$ and $Y$ if $X$ and $Y$ partition $V(G)$
and every edge of $G$ has one endpoint in $X$ and the other in $Y$.

\begin{lemma}[Pokrovskiy, Versteegen and Williams~\cite{PokrovskiyVersteegenWilliams2026}]
\label{lem:bip-path}
Let $G$ be a bipartite graph with parts $X$ and $Y$.  If every vertex of $Y$
has degree at least $(|X|+|Y|)/2$, then $G$ contains a path covering $2|Y|$ vertices.
\end{lemma}

\begin{lemma}[Pokrovskiy, Versteegen and Williams~\cite{PokrovskiyVersteegenWilliams2026}]
\label{lem:bip-cover}
 Let $G$ be a bipartite graph with parts $X$ and $Y$. Define $X_0=\{x\in X:d_G(x)=|Y|\}$, $X_1=X\setminus X_0$, $Y_0=\{y\in Y:d_G(y)=|X|\}\neq\emptyset$ and $Y_1=Y\setminus Y_0$. Suppose that

(i) $|X|>|Y|$, and

(ii) $X_1=Y_1=\emptyset$ or $|X_0|/|Y_1|>2|X_1|/|Y_0|$.

\noindent Then $G$ contains a collection of at most $\ceil{|X|/(|Y|+1)}$ paths whose union covers
$V(G)$.
\end{lemma}

\begin{lemma}[Pokrovskiy, Versteegen and Williams~\cite{PokrovskiyVersteegenWilliams2026}]
\label{lem:blue-cover}
Let $K_n$ be red--blue edge-colored, and let $P$ be a blue path in $K_n$. Let $X=V(P)$ and $Y=V(K_n)\setminus X$.  
Let $Y_0$ be the set of vertices in $Y$ that  are not contained in a blue edge with any vertex in $X$, and let $Y_1=Y\setminus Y_0$.  Then $K_n$ admits a blue path cover of size at most \[1+\left\lceil \frac{|Y_1|}{2}\right\rceil+|Y_0|.\]
\end{lemma}

\section{Proof of Theorem~\ref{thm:main}}

Let $\mathcal X_n$ be the set of all red--blue edge-colorings of $K_n$.  For
$\chi\in\mathcal X_n$, let $g(n,\chi)$ be the minimum size of a monochromatic
path cover whose paths all have the same color, and let $g(n)=\max_{\chi\in\mathcal X_n}g(n,\chi)$.

Suppose  that Theorem~\ref{thm:main} is false.  Choose a
counterexample $(G,\chi)$ with $G=K_n$ and $n$ minimal. Let \[c=\lfloor\sqrt{n}\rfloor~~~~~\textup{and}~~~~~n=c^2+r,\] where $c\ge 1$ and $0\le r\le 2c$. Thus $g(n,\chi)>c$, whereas
\begin{equation}
   g(m)\le\floor{\sqrt m}\qquad\text{for every }1\le m<n.
   \label{eq:minimality}
\end{equation}
For each $v\in V(G)$, let $\Nb(v)$ and $\Nr(v)$ denote the blue and red neighbourhoods of $v$, respectively. More precisely,
\[
\Nb(v)=\{x\in N_G(v):\chi(xv)=\textup{blue}\}
\quad\text{and}\quad
\Nr(v)=\{x\in N_G(v):\chi(xv)=\textup{red}\}.
\]

If $\chi$ is a cut coloring, choose $Q$ to be a longest monochromatic path with this property; otherwise, let $Q$ be any longest monochromatic path in $G$. By symmetry of red and blue, we may assume that $Q$ is blue. Write $Q=v_1v_2\cdots v_\ell$. Let $X=V(Q)$, $Y=V(G)\setminus X$ and $w=|Y|$. Partition $Y$ into $Y_0=\{u\in Y:N^b(u)\cap X=\emptyset\}$ and $Y_1=Y\setminus Y_0$. Let $a_0=|Y_0|$ and $a_1=|Y_1|=w-a_0$. 
\vskip 2mm
Before starting to prove Theorem \ref{thm:main}, we discuss some properties of the minimal counterexample, that is, Lemmas \ref{lem:intersection}-\ref{lem:residual}.

We first use minimality to bound intersections between red and blue paths. 
%This bound will be used throughout the proof.

\begin{lemma}
\label{lem:intersection}
If $P_1$ is a red path and $P_2$ is a blue path in $G$,
then $|V(P_1)\cap V(P_2)|\le r$.
\end{lemma}

\begin{proof}
Let $S=V(P_1)\cap V(P_2)$ and suppose that
$|S|\ge r+1$.  Then
\[
   |V(G)\setminus S|\le c^2-1.
\]
If $S=V(G)$, either $P_1$ or $P_2$ alone covers $G$, a
contradiction.  If $V(G)\setminus S\ne\emptyset$, then let $H=G[V(G)\setminus S]$ and $m=|V(H)|$.  By
\eqref{eq:minimality}, $H$ has a same-color monochromatic path cover
$\mathcal C$ of size at most $\floor{\sqrt m}\le c-1$.  If $\mathcal C$ is
red, add $P_1$; if it is blue, add $P_2$.  In either case
we obtain a same-color path cover of $G$ with at most $c$ paths, again a
contradiction.
\end{proof}

Next we estimate the bounds for $w$, $a_0$, and $a_1$.

\begin{lemma}
\label{lem:basic-bounds}
We have $c\le w\le r-2$, $a_1\ge1$ and $a_0\ge 2c-1-w\ge1$.
\end{lemma}

\begin{proof}
If $w\le c-1$, then $Q$, together with the $w$ singleton paths corresponding
to the vertices of $Y$, is a blue path cover of size at most $c$.  Hence
$w\ge c$. 

By Lemma~\ref{lem:long-path}, $w\le n-\floor{2n/3}-1$. Writing $n=3p$, $3p+1$, or $3p+2$ shows in each case that
$w<\ceil{|X|/2}$.  If $\chi$ is not a cut coloring, then applying
Lemma~\ref{lem:EG-extension}  with $A=X$ and $B=Y$ gives a red path
that contains $Y$ and $w+2$ vertices of $X$.  Its intersection with $Q$ has
size at least $w+2$, and hence Lemma~\ref{lem:intersection} yields $w\le r-2$.
If $\chi$ is a cut coloring, then by the
choice of $Q$, the edge $v_\ell v_1$ is blue. So $Qv_1$ is a blue cycle.  The
maximality of $Q$ implies that every edge between $X$ and $Y$ is red.  Set
$t=\ceil{|X|/(w+1)}$ and partition $X$ into sets $X_1,\ldots,X_t$ such that
$|X_1|=w+1$ and $|X_i|\le w+1$ for every $2\le i\le t$.  Such a partition
exists because $w<\ceil{|X|/2}$ implies $|X|\ge w+1$.  Write
$X_1=\{x_0,x_1,\ldots,x_w\}$ and $Y=\{y_1,\ldots,y_w\}$.  Then $x_0y_1x_1y_2\cdots y_wx_w$ is a red path covering $X_1\cup Y$.  Each remaining set $X_i$ can be covered
by a red path alternating between $X_i$ and $|X_i|-1$ distinct vertices of
$Y$; vertices of $Y$ may be reused by different paths because paths in a
cover are allowed to intersect.  Moreover,
\[
  |X|=n-w\le c^2+2c-w\le c(w+1),
\]
where the last inequality is equivalent to $(c+1)(w-c)\ge0$.  We therefore
obtain a red path cover with at most $c$ paths, a contradiction.  Therefore,
$w\le r-2$ in all cases.

By Lemma~\ref{lem:blue-cover} and the choice of $(G,\chi)$, we have $1+\ceil{a_1/2}+a_0\ge c+1$.
Since $a_1=w-a_0$, this implies $w+a_0\ge2c-1$, and hence
$a_0\ge2c-1-w$.  Combining this with $w\le r-2\le2c-2$ gives $a_0\ge1$.

Finally, if $a_1=0$, then every edge between $X$ and $Y$ is red. The same
alternating-path argument used above gives a red path cover with at most
$c$ paths.  Hence, $a_1\ge1$.
\end{proof}

A set $S\subseteq V(G)$ is called a \emph{red clique} if every pair of distinct vertices in $S$ is joined by a red edge. We next use the maximality of $Q$ to find a red clique. This also bounds blue degrees from $Y_1$ into $X$.

\begin{lemma}
\label{lem:red-clique}
For each $y\in Y_1$, the graph $G[X]$ contains a red clique of order
\[|\Nb(y)\cap X|+1.\]  Furthermore, $|\Nb(y)\cap X|\le r-2$.
\end{lemma}

\begin{proof}
Let $U=\Nb(y)\cap X$.  The maximality of $Q$ implies that
$v_1,v_\ell\notin U$ and that $U$ contains no two consecutive vertices of
$Q$.  If $v_i,v_j\in U$ with $i<j$ and $v_{i-1}v_{j-1}$ were blue, then
\[
 v_1Qv_{i-1}v_{j-1}Qv_i\,y\,v_jQv_\ell
\]
would be a blue path containing $X\cup\{y\}$, a contradiction.  Hence all
edges between predecessors of vertices in $U$ are red.  Similarly,
$v_\ell v_{i-1}$ is red for every $v_i\in U$; otherwise $v_1Qv_{i-1}v_\ell Qv_i\,y$  would be a blue path on $|X|+1$ vertices.  Consequently, \[S_y=\{v_\ell\}\cup\{v_{i-1}:v_i\in U\}\] is a red clique of order $|U|+1$.

If $|U|\ge r$, then $S_y$ contains a red path on $r+1$ vertices, contradicting
Lemma~\ref{lem:intersection}.  Thus $|U|\le r-1$.  Suppose that equality
holds.  Then $|S_y|=r$. By Lemma \ref{lem:basic-bounds}, we have
\[
   |X\setminus S_y|=c^2-w\ge c^2-(2c-2)=(c-1)^2+1>0.
\]
Choose $x_0\in X\setminus S_y$ and $y_0\in Y_0$, the latter being available
by Lemma~\ref{lem:basic-bounds}.  A red Hamilton path of $S_y$ ending at
$v_\ell$ extends through $v_\ell y_0x_0$ to a red path containing $r+1$
vertices of $X$.  This again contradicts Lemma~\ref{lem:intersection}. Hence, $|\Nb(y)\cap X|\le r-2$.
\end{proof}
In what follows, for each $y\in Y_1$, the notation $S_y$ always refers to the
particular red clique constructed in the preceding proof, namely
\[
S_y=\{v_\ell\}\cup\{v_{i-1}:2\le i\le\ell,\ v_i\in\Nb(y)\cap X\}.
\]
Define
\[
   \mu=\max_{y\in Y_1}|\Nb(y)\cap X|.
\]
Then $1\le\mu\le r-2$, and every $y\in Y$ has at least $|X|-\mu$ red
neighbours in $X$.

The following lemma records a suitable alternating extension of this red clique.
%We extend the red clique by a suitable alternating sequence. And we record this argument next.

\begin{lemma}
\label{lem:splice}
Let $y\in Y_1$ and $\sigma=|\Nb(y)\cap X|$. Let $S_y$ denote the red clique associated with $y$ in
Lemma~\ref{lem:red-clique}. Define $W=X\setminus S_y$ and $a=r-\sigma$.
Suppose that there exist pairwise distinct vertices
$y_1,\ldots,y_a\in Y$ and pairwise distinct vertices
$x_1,\ldots,x_a\in W$ such that $x_j\in\Nr(y_j)\cap\Nr(y_{j+1})\cap W$ for every $1\le j<a$, and $x_a\in\Nr(y_a)\cap W$.
Then $(G,\chi)$ is not a counterexample.
\end{lemma}

\begin{proof}
Suppose that the hypotheses of Lemma \ref{lem:splice} hold. The sequence $y_1x_1y_2x_2\cdots y_ax_a$ is a red path.  Since $S_y$ is a red clique containing $v_\ell$, it has a
red Hamilton path ending at $v_\ell$.  The edge $v_\ell y_1$ is red by the
maximality of $Q$, so the two paths can be concatenated.  The resulting red
path contains $|S_y|+a=r+1$ vertices of $X$, contrary to Lemma~\ref{lem:intersection}.
\end{proof}

The following lemma provides a
method to cover subsets of $X$ by a few red paths.

\begin{lemma}
\label{lem:residual}
Let $D\subseteq X$ and $d=|D|$. Let $h\ge1$ be the number of red paths
available to cover $D$.  Let $p=d-h(a_0+1)$ and $q=\min\{p,h\}$. 
Then $D$ can be covered by at most $h$ red paths if at least one of the
following conditions holds:
\begin{enumerate}[label=\textup{(\roman*)}]
\setlength{\itemsep}{1pt}
\setlength{\parsep}{1pt}
\setlength{\parskip}{1pt}
\item $p\le0$;
\item $0<p\le\min\{h,\floor{(d-\mu)/2}\}$;
\item $p>0$, $p\le q|Y_1|$, $d-2\mu\ge p-q$, and
      $d-\mu\ge p+q$.
\end{enumerate}
\end{lemma}

\begin{proof}
Every set of at most $a_0+1$ vertices of $X$ lies on a red path alternating
between that set and $Y_0$.  The vertices of $Y_0$ may be reused in different
paths.  Thus, if $p\le0$, partitioning $D$ into at most $h$ such sets proves
(i).

Suppose (ii) holds.  Choose $z\in Y_1$.  Since
$|\Nr(z)\cap D|\ge d-\mu\ge2p$, choose $2p$ distinct vertices of
$\Nr(z)\cap D$ and group them into $p$ pairs.  Partition $D$ into $p$ sets of
size $a_0+2$, each containing one designated pair, and $h-p$ sets of size
$a_0+1$.  A set of the first kind, say
$\{u,v,x_1,\ldots,x_{a_0}\}$ with $u,v\in\Nr(z)$, is covered by $u z v y_1x_1y_2x_2\cdots y_{a_0}x_{a_0}$, where $Y_0=\{y_1,\ldots,y_{a_0}\}$.  Sets of the second kind are handled as
in (i).  Hence $h$ red paths cover $D$.

Assume (iii) holds.  Since $q\le p\le q|Y_1|$, there exist integers
$b_1,\ldots,b_q$ such that \[\sum_{i=1}^q b_i=p~~~~~\textup{and}~~~~~1\le b_i\le|Y_1|.\]
For each $i\in\{1,\ldots,q\}$, choose an ordered list of distinct vertices
$y_{i,1},\ldots,y_{i,b_i}$ of $Y_1$; lists belonging to different values of
$i$ may overlap. For each $i\in\{1,\ldots,q\}$ and each $j\in\{1,\ldots,b_i-1\}$, let $C_{i,j}=\Nr(y_{i,j})\cap\Nr(y_{i,j+1})\cap D$. Set $C_{i,0}=\Nr(y_{i,1})\cap D$ and
$C_{i,b_i}=\Nr(y_{i,b_i})\cap D$. Note that for any vertex $y\in Y$, $|\Nr(y)\cap D|\ge |D|-\mu\ge d-\mu$.  The inclusion--exclusion principle gives \[|C_{i,j}|\ge |N^r(y_{i,j})\cap D|+| N^r(y_{i,j+1})\cap D|-|D|\ge d-2\mu\ge p-q\] for each $i\in\{1,\ldots,q\}$ and each $j\in\{1,\ldots,b_i-1\}$.  Moreover, we have $|C_{i,0}|,|C_{i,b_i}|\ge d-\mu\ge p+q$ for each $i\in\{1,\ldots,q\}$.

We now construct $q$ red paths whose vertices alternate between $D$ and $Y_1$. For each $i\in\{1,\ldots,q\}$, the $i$-th path will contain the vertices $y_{i,1},y_{i,2},\ldots,y_{i,b_i}$ and have order $2b_i+1$. These paths require a total of $\sum_{i=1}^q(b_i-1)=p-q$ internal vertices from $D$ and $2q$ endpoints from $D$. The bounds established in the preceding paragraph allow us to choose the required vertices greedily. First, choose $u_{i,j}\in C_{i,j}$ for every $i\in\{1,\ldots,q\}$ and $j\in\{1,\ldots,b_i-1\}$; these vertices will serve as the internal vertices from $D$. Next, again greedily, choose the endpoints $u_{i,0}\in C_{i,0}$ and $u_{i,b_i}\in C_{i,b_i}$ for each $i\in\{1,\ldots,q\}$. In this way, all the chosen vertices $u_{i,j}$ are pairwise distinct. Consequently, for each $i\in\{1,\ldots,q\}$, $P_i=u_{i,0}y_{i,1}u_{i,1}\cdots y_{i,b_i}u_{i,b_i}$ is a red path. Collectively, these paths contain exactly $p+q$ vertices of $D$.

For each $i\in\{1,\ldots,q\}$, extend $P_i$ from its last endpoint by alternating between the $a_0$ vertices of $Y_0$ and $a_0$ previously unused vertices of $D$. After these extensions, the number of uncovered vertices of $D$ is $d-(p+q)-qa_0=(h-q)(a_0+1)$. Partition the remaining vertices into $h-q$ sets of size $a_0+1$, and cover each set using the construction in part~(i). Together with the $q$ extended paths, the resulting $h$ red paths cover $D$.
\end{proof}

\vskip 3mm
We now begin to prove Theorem \ref{thm:main}. Let
\begin{equation}
   t=1+\ceil{\frac{a_1}{2}}+a_0-c.
   \label{eq:t}
\end{equation}
By Lemma~\ref{lem:blue-cover} and the counterexample assumption, $t\ge1$.

We consider the following two cases separately. 
\vskip 2mm
\noindent{\bf Case 1. $c\ge 4$.}
\vskip 2mm
%\subsection{Excluding \texorpdfstring{$c\ge4$}{c >= 4}}

%We now combine previous lemmas to exclude $c\ge4$.

%\begin{lemma}
%\label{lem:c-at-most-three}
%We have $c\le3$.
%\end{lemma}

%\begin{proof}
%Suppose that $c\ge4$. 

We first show the following claim.

\begin{claim}
\label{cl:t-one}
$t=1$.
\end{claim}

\begin{proof}
Suppose that $t\ge2$. By~(2), $a_0+\ceil{a_1/2}\ge c+1$. Since $\ceil{a_1/2}\le(a_1+1)/2$ and $w=a_0+a_1$, we obtain $a_0\ge 2c+1-w.$ Together with Lemma~\ref{lem:basic-bounds}, which gives $w\le r-2\le2c-2$, this yields $a_0\ge3$. By Lemmas \ref{lem:basic-bounds} and \ref{lem:red-clique}, we have 
\[|X|-w-2\mu=c^2+r-2w-2\mu\ge c^2-3r+8\ge c^2-6c+8\ge0.\]
Thus $|X|-\mu\ge(|X|+|Y|)/2$.  Every vertex of $Y$ has at least $|X|-\mu$ red neighbours in $X$. Applying Lemma~\ref{lem:bip-path} to the bipartite graph with parts $X$ and $Y$ whose edge set consists precisely of the red edges of $G$ joining $X$ and $Y$, we obtain a red path $P$ containing every vertex of $Y$ and exactly $w$ vertices of $X$. Let $D=X\setminus V(P)$, $d=|D|=c^2+r-2w$ and $h=c-1$. Define $p=d-h(a_0+1)$.

If $p\le0$, Lemma~\ref{lem:residual}(i) covers $D$ with $h$ red paths. Together with $P$, these $h$ red paths form a red path cover of $G$ of size at most $c$. Then $g(n,\chi)\le c$, a contradiction. If $0<p\le h$, then $d\le h(a_0+2)$. By Lemmas \ref{lem:basic-bounds} and \ref{lem:red-clique},  we have
\[d-\mu-2p=2h(a_0+1)-\mu-d\ge h a_0-(r-2)\ge3(c-1)-(2c-2)>0.\]
Hence Lemma~\ref{lem:residual}(ii) gives the same contradiction.

It remains to consider $p>h$.  Here $q=h$.  Since
$a_0\ge2c+1-w$ and $a_0\le w$, we have $w\ge c+1$, and therefore
$d\le c^2-2$.  As $a_1=w-a_0$,
\[
q|Y_1|-p=h(w+1)-d
   \ge(c-1)(c+2)-(c^2-2)>0.
\]
Moreover, $\mu\le r-2\le2c-2=2h$, so
\begin{align*}
d-2\mu-(p-q)&=h(a_0+2)-2\mu\ge h(a_0-2)\ge0,\\
d-\mu-(p+q)&=h a_0-\mu\ge h(a_0-2)\ge0.
\end{align*}
All hypotheses of Lemma~\ref{lem:residual}(iii) hold, again producing a red
path cover of size at most $c$.  So $t=1$.
\end{proof}

By Claim~\ref{cl:t-one},
\begin{equation}
   c=a_0+\ceil{\frac{a_1}{2}}.
   \label{eq:t-one-relation}
\end{equation}
We now divide the argument according to whether the largest blue neighbourhood
from $Y_1$ into $X$ is below or above the natural threshold $r/2$.

\medskip
\noindent\textbf{Subcase 1. $2\mu\le r$.} 
\vskip 2mm
Choose $z\in Y_1$ with $|\Nb(z)\cap X|=\mu$. Let $S_z$ be the red clique associated with $z$ by Lemma~\ref{lem:red-clique}. Define $W=X\setminus S_z$ and $a=r-\mu$. For each $y\in Y$, let $R_y=\Nr(y)\cap W$.  Since $\mu\le r/2\le c$,
\begin{align}
|R_y|-a
 &\ge |W|-\mu-(r-\mu)
   =c^2-w-\mu-1
   \ge c^2-3c+1>0,                                      \label{eq:single-R}\\
|R_y\cap R_{y'}|-(a-1)
 &\ge |W|-2\mu-(r-\mu-1)
   =c^2-w-2\mu
   \ge c^2-4c+2\ge0                                    \label{eq:double-R}
\end{align}
for all $y,y'\in Y$.

If $a\le w$, choose distinct $y_1,\ldots,y_a\in Y$. It follows from \eqref{eq:double-R} that $|R_{y_i}\cap R_{y_{i+1}}|\ge a-1$ for every $i\in\{1,\dots,a-1\}$. Moreover, \eqref{eq:single-R} yields $|R_{y_a}|\ge a$.  Choosing representatives greedily gives distinct $x_1,\ldots,x_a\in W$ satisfying Lemma~\ref{lem:splice}, a contradiction.

We may therefore assume that $a\ge w+1$.  This gives
$\mu\le r-w-1$, and Lemma~\ref{lem:basic-bounds} yields $a_0-\mu\ge2c-r\ge0$. Let $Y=\{u_1,\ldots,u_w\}$. For every $i\in\{1,\ldots,w-1\}$, \eqref{eq:double-R} yields $|R_{u_i}\cap R_{u_{i+1}}|\ge a-1\ge w$. Moreover, it follows from \eqref{eq:single-R} that $|R_{u_w}|\ge a\ge w+1$. Choose distinct $b_i\in R_{u_i}\cap R_{u_{i+1}}$ for $i<w$ and
$b_w\in R_{u_w}$. By \eqref{eq:single-R}, $|R_{u_1}|\ge a+1\ge w+2$, so we may then choose
$b_0\in R_{u_1}\setminus\{b_1,\ldots,b_w\}$.   Hence $P'=b_0u_1b_1u_2b_2\cdots u_wb_w$ is a red path covering all of $Y$ and $w+1$ vertices of $X$.

Let $D=X\setminus V(P')$,  $d=c^2+r-2w-1$ and $h=c-1$. Define $p=d-h(a_0+1)$.  If $p\le0$, apply
Lemma~\ref{lem:residual}(i).  If $0<p\le h$, then
$d\le h(a_0+2)$ and
\[
d-\mu-2p=2h(a_0+1)-\mu-d\ge h a_0-\mu\ge0,
\]
where we used $a_0\ge\mu$; thus Lemma~\ref{lem:residual}(ii) applies. Finally, suppose $p>h$ and set $q=h$.  By Lemma \ref{lem:basic-bounds}, we have $w\ge c$ and $d\le c^2-1$. So 
\[
q|Y_1|-p=h(w+1)-d\ge(c-1)(c+1)-(c^2-1)=0.
\]
Since $2\mu\le r\le 2c$ and $c\ge4$, we have $\mu\le c\le2c-2=2h$. Together with $a_0\ge\mu$, this implies that
\[d-2\mu-(p-q)=h(a_0+2)-2\mu\ge0~\textup{and}~d-\mu-(p+q)=h a_0-\mu\ge0.\]
Hence Lemma~\ref{lem:residual}(iii) applies. In every subcase, $P'$ and the
resulting $h$ red paths cover all vertices of $G$, contradicting the choice of $(G,\chi)$.
Thus, $2\mu\le r$ is impossible.

\medskip
\noindent\textbf{Subcase 2. $2\mu\ge r+1$.}
\vskip 2mm
Call a vertex $y\in Y_1$ \emph{high} if $2|\Nb(y)\cap X|\ge r+1$. For each $y\in Y_1$, define \[I(y)=\{i:v_i\in\Nb(y)~\textup{and}~i\in\{1,\ldots,\ell\}\}.\]
Let $\mathcal H$ be the $3$-uniform hypergraph on $Y_1$ in which
$\{u,x,y\}$ is an edge if, after possibly reordering these three vertices,
there are indices $\alpha\le \beta<\gamma\le \delta$ such that \[\alpha\in I(u),~~~~~\beta,\gamma\in I(x)~~~~~\textup{and}~~~~~\delta\in I(y).\]  Such an edge will be
called a \emph{triple}.  A set $Y'\subseteq Y_1$ is called \emph{triple-free} if $E(\mathcal H[Y'])=\emptyset$. For such a triple, $u v_\alpha Qv_\beta xv_\gamma Qv_\delta y$ is a blue path containing its three vertices.

By Claim~\ref{cl:t-one}, the upper bound supplied by
Lemma~\ref{lem:blue-cover} is $c+1$. Moreover, any two distinct vertices $y,y'\in Y_1$ can be covered by a blue path whose internal vertices lie on $Q$: choose blue neighbours $x,x'\in X=V(Q)$ of $y$ and $y'$, respectively, and use $yxy'$ if $x=x'$, or $yxQx'y'$ otherwise.  If $a_1$ is odd and
$E(\mathcal H)\ne\varnothing$, one triple can be covered by one blue path and
the remaining $a_1-3$ vertices of $Y_1$ can be paired, saving one path.  This
would give a blue path cover of size at most $c$.  Hence
\begin{equation}
   a_1\text{ odd}\quad\Longrightarrow\quad E(\mathcal H)=\varnothing.
   \label{eq:odd-triple-free}
\end{equation}
If $a_1$ is even and $\mathcal H$ contains two vertex-disjoint edges, then the two
triples can be covered by two blue paths. The remaining $a_1-6$ vertices can then be paired, saving one path. Therefore
\begin{equation}
   a_1\text{ even}\quad\Longrightarrow\quad \nu(\mathcal H)\le1,
   \label{eq:matching-one}
\end{equation}
where $\nu(\mathcal H)$ is the matching number of $\mathcal H$.

We next establish the structural estimate needed for every triple-free set
that contains a high vertex.

\begin{claim}
\label{cl:interval}
Let $Y'\subseteq Y_1$ be triple-free and suppose that $s\in Y'$ is high.  Let
$\mu_s=|I(s)|$.  Then
\[
   a_0+\max\{|Y'|-2,0\}<r-\mu_s.
   \label{eq:interval-conclusion}
\]
\end{claim}

\begin{proof}
 Let $k_0=0$, $k_{\mu_s+1}=\ell+1$ and $I(s)=\{k_1,k_2,\ldots,k_{\mu_s}\}$, where $k_1<k_2<\ldots<k_{\mu_s}$. For every $i\in\{0,1,\ldots,\mu_s\}$, define \[M_i=\{y\in Y'\setminus\{s\}:\min I(y)\le k_i\}~\textup{and}~N_{i+1}=\{y\in Y'\setminus\{s\}:\max I(y)\ge k_{i+1}\}.\] 
If $1\le i<j\le\mu_s$, there cannot be distinct vertices $z_1\in M_i$ and $z_2\in N_j$. Otherwise, suppose that $z_1\in M_i$ and $z_2\in N_j$ are distinct. Then $z_1,z_2\in Y'$. Choose $b_1\in I(z_1)$ with $b_1\le k_i$ and $b_2\in I(z_2)$ with $b_2\ge k_j$. Since $k_i, k_j\in I(s)$ and $b_1\le k_i<k_j\le b_2$, $\{z_1,s,z_2\}\subseteq Y'$ is a triple, a contradiction.

Choose $e\in\{0,\ldots,\mu_s-1\}$ maximal subject to $|M_e|\le1$; such an
$e$ exists because $M_0=\varnothing$.  We have $N_{e+2}=\varnothing$.
Indeed, if $e<\mu_s-1$ and $N_{e+2}\ne\varnothing$, maximality gives
$|M_{e+1}|\ge2$, contradicting the preceding paragraph; if
$e=\mu_s-1$, the assertion is simply $N_{\mu_s+1}=\varnothing$.

Let $J=Y'\setminus(M_e\cup\{s\})$. If $J=\varnothing$, let $F_0=\varnothing$.  Otherwise, every blue neighbour
in $X$ of every vertex of $J$ lies on the subpath $Q'=v_{k_e+1}Qv_{k_{e+2}-1}$, and we let $F_0=V(Q')$.  Let $S_s$ be the red clique associated with $s$ by Lemma~\ref{lem:red-clique}. Let $W=X\setminus S_s$ and $F=W\setminus F_0$.  By construction, all edges between $J\cup Y_0$ and $F$ are red.

We claim that $|F|\ge r-\mu_s$.  Suppose that $J=\varnothing$. By Lemma \ref{lem:basic-bounds}, $w\le r-2\le 2c-2$. It follows that \[|F|=|X|-|S_s| =c^2+r-w-\mu_s-1\ge(c-1)^2+r-\mu_s\ge r-\mu_s.\] Suppose that $J\ne\varnothing$.  The maximality of $Q$ implies $k_1\ge2$, $k_{\mu_s}\le\ell-1$, and $k_{i+1}-k_i\ge2$ for $1\le i<\mu_s$.  By Lemma~\ref{lem:red-clique}, $S_s\cap F_0=\{v_{k_{e+1}-1},v_{k_{e+2}-1}\}.$
It follows that $|F|=\ell-\mu_s+2-(k_{e+2}-k_e)$.  All $\mu_s+1$ gaps between $k_0,k_1,\ldots,k_{\mu_s+1}$ have size at least
$2$.  Hence $k_{e+2}-k_e\le\ell-2\mu_s+3$, and therefore $|F|\ge\mu_s-1\ge r-\mu_s$, the last inequality following
from the fact that $s$ is high.

Finally, $|M_e|\le1$ gives $|J|\ge\max\{|Y'|-2,0\}$. If $a_0+\max\{|Y'|-2,0\}\ge r-\mu_s$, we could choose
$r-\mu_s$ distinct vertices from $J\cup Y_0$ and the same number of distinct
vertices from $F$.  The complete red bipartite graph between these two sets
would satisfy the hypothesis of Lemma~\ref{lem:splice}, a contradiction.
This proves the claim.
\end{proof}

We now apply Claim~\ref{cl:interval}.  If $a_1$ is odd, then
$Y'=Y_1$ is triple-free by \eqref{eq:odd-triple-free} and contains a high
vertex.  Let $a_1=2m+1$.  By \eqref{eq:t-one-relation}, we have $a_0=c-m-1$ and $w=c+m$. 
If $a_1=1$, then $a_0=c-1$; if $a_1\ge3$, then $a_0+a_1-2=w-2\ge c-1$.  In both cases,
\[
 a_0+\max\{a_1-2,0\}\ge c-1\ge r-\mu_s,
\]
because a high vertex satisfies $r-\mu_s\le\floor{(r-1)/2}\le c-1$.
This contradicts Claim~\ref{cl:interval}.

If $a_1$ is even and $E(\mathcal H)=\varnothing$, the same argument applies
with $a_1=2m\ge2$. Then $a_0=c-m$, $w=c+m$, and
\[
a_0+a_1-2=w-2\ge c-1\ge r-\mu_s,
\]
again contradicting Claim~\ref{cl:interval}.  Consequently,
\begin{equation}
   a_1\text{ is even and }E(\mathcal H)\ne\varnothing.
   \label{eq:even-nonempty}
\end{equation}

\begin{claim}
\label{cl:no-edge-all-high}
No edge of $\mathcal H$ contains every high vertex of $Y_1$.
\end{claim}

\begin{proof}
Suppose that an edge $T$ contains every high vertex, and choose a high vertex
$s'\in T$.  Let $\mu'=|I(s')|,$ $a'=r-\mu'$ and $\eta=\floor{r/2}$.
Let $C=S_{s'}$ be the red clique from Lemma~\ref{lem:red-clique} and set
$W'=X\setminus C$ and $V_1=Y\setminus T$.  Since $a_1$ is even and
$T\in E(\mathcal H)$, we have $a_1\ge4$.  By
\eqref{eq:t-one-relation},
\[
   |V_1|=a_0+a_1-3=c+\frac{a_1}{2}-3\ge c-1\ge a'.
\]
Note that every vertex of $V_1$ has at most $\eta$ blue neighbours in $X$.  Hence
each vertex of $V_1$ has at least $|W'|-\eta$ red neighbours in $W'$, and any
two vertices of $V_1$ have at least $|W'|-2\eta$ common red neighbours there.
Furthermore,
\begin{align*}
|W'|-\eta-a'&=c^2-w-\eta-1
   \ge c^2-r-\eta+1\ge0,\\
|W'|-2\eta-(a'-1)&=c^2-w-2\eta
   \ge c^2-r-2\eta+2\ge0,
\end{align*}
where we used $w\le r-2$, $r\le2c$, and $c\ge4$.  We may therefore choose
$a'$ distinct vertices of $V_1$ and greedily select the distinct vertices required
by Lemma~\ref{lem:splice}, with $S_{s'}=C$ and $W=W'$.  This contradiction
proves the claim.
\end{proof}

Choose an edge $T\in E(\mathcal H)$.  By
Claim~\ref{cl:no-edge-all-high}, the set $Y_1\setminus T$ contains a high
vertex.  Choose $s\in Y_1\setminus T$ so that $|I(s)|$ is maximum over
$Y_1\setminus T$. Let $\mu_s=|I(s)|$ and $Y'=Y_1\setminus T$. By \eqref{eq:matching-one}, $Y'$ is triple-free.  In particular, every vertex of $Y'$ has at most $\mu_s$ blue neighbours in $X$.  We may therefore apply Claim~\ref{cl:interval} to $Y'$ and $s$. Suppose that $a_1\ge8$. Then
\[
a_0+|Y'|-2
 =a_0+a_1-5
 =c+\frac{a_1}{2}-5
 \ge c-1
 \ge r-\mu_s,
\]
contradicting Claim~\ref{cl:interval}. Hence, by \eqref{eq:even-nonempty}, $a_1=2\lambda$ for some
$\lambda\in\{2,3\}$. Finally, \eqref{eq:t-one-relation} yields $a_0=c-\lambda$ and $w=c+\lambda$.

Let $S_s$ be the red clique associated with $s$ by Lemma \ref{lem:red-clique}. Let
$W=X\setminus S_s$ and $b=r-\mu_s$.  Since $c\ge 4$ and $\lambda\le 3$, we have $|W|-b=c^2-c-\lambda-1\ge0$.
If $a_0\ge b$, choose $b$ vertices of $Y_0$ and $b$ vertices of $W$. Since
all edges between $Y_0$ and $W$ are red, Lemma~\ref{lem:splice} gives a contradiction. Now assume that $a_0<b$.  Since $b\le c-1$ and $\lambda\le 3$, we have \[b-a_0\le(c-1)-(c-\lambda)=\lambda-1\le2.\] 
Choose all $a_0$ vertices of $Y_0$ and $b-a_0$ vertices of $Y'$.  There are
enough vertices in $Y'$, since $|Y'|=2\lambda-3$ and
$b-a_0\le\lambda-1\le2\lambda-3$.  Order the resulting $b$ vertices as
$z_1,\ldots,z_b$ so that no two consecutive vertices belong to $Y'$.  This
is possible because $a_0=c-\lambda\ge1$ and $b-a_0\le2$.  Every vertex of
$Y'$ has at most $\mu_s$ blue neighbours in $X$.  Therefore, for every
consecutive pair, and also for the last vertex, the relevant red
neighbourhood in $W$ has size at least $|W|-\mu_s$.  Since $\mu_s\le r-2\le 2c-2$, we have
\[
|W|-\mu_s-b
 =c^2-c-\lambda-\mu_s-1
 \ge c^2-3c-\lambda+1\ge0.
\]
Thus the distinct vertices required by Lemma~\ref{lem:splice} can be chosen
greedily. Therefore, $2\mu\ge r+1$ is also impossible.

Both cases are impossible, so the assumption $c\ge4$ cannot hold.
%\end{proof}

\vskip 2mm
\noindent{\bf Case 2. $c\le 3$}
\vskip 2mm
%\subsection{The remaining values of \texorpdfstring{$c$}{c}}

%It remains to exclude $1\le c\le 3$. The next lemma handles these cases.

%\begin{lemma}
%\label{lem:small-c}
%No minimal counterexample exists with $1\le c\le3$.
%\end{lemma}

%\begin{proof}
By Lemmas~\ref{lem:long-path}, \ref{lem:basic-bounds} and~\ref{lem:red-clique}, we obtain 
\begin{equation}
c\le w\le r-2,\qquad r\le2c,\qquad
w\le n-\floor{\frac{2n}{3}}-1,                         \label{eq:small-basic}
\end{equation}
together with
\begin{equation}
   1\le\mu\le r-2.                                      \label{eq:small-mu}
\end{equation}
It follows from \eqref{eq:small-mu} that $r\ge3$. 

If $c=1$, then \eqref{eq:small-basic} gives $r\le2c=2$, a contradiction. Hence, $2\le c\le 3$.

Suppose that $c=2$. It follows from \eqref{eq:small-basic} that $r=4$ and $w=2$. Lemma~\ref{lem:basic-bounds} then yields $a_0=a_1=1$. Let $Y_1=\{z\}$ and $Y_0=\{y_0\}$. Thus, $\mu=|\Nb(z)\cap X|$.
Let $S_z$ denote the red clique associated with $z$ as in Lemma~\ref{lem:red-clique}, and set $W=X\setminus S_z$. Since $|X|=6$, we have $|W|=5-\mu$. Finally, \eqref{eq:small-mu} gives $\mu\in\{1,2\}$. If $\mu=2$, then $|\Nr(z)\cap W|\ge|W|-\mu\ge1$.  Choose
$x_1\in\Nr(z)\cap W$ and $x_2\in W\setminus\{x_1\}$.  Since every edge from
$y_0$ to $X$ is red, the vertices $z,y_0,x_1,x_2$ satisfy Lemma~\ref{lem:splice} with $a=r-\mu=2$, a contradiction.

Suppose that $\mu=1$.  Since $|\Nr(z)\cap X|=|X|-\mu=5$, choose $x_1\in\Nr(z)\cap X$ and
$x_2\in X\setminus\{x_1\}$.  Since every edge from $y_0$ to $X$ is red,
$P_1=zx_1y_0x_2$ is a red path. Moreover,
$|\Nr(z)\cap W|\ge|W|-\mu\ge3$, so we may choose
$x_0\in(\Nr(z)\cap W)\setminus\{x_1,x_2\}$.  Then
$P_2=x_0zx_1y_0x_2$ is red.  Let
$D=X\setminus\{x_0,x_1,x_2\}$.  For $d=3$, $h=1$, $a_0=1$, and $\mu=1$, we
have \[p=d-h(a_0+1)=1\le\min\{h,\floor{\frac{d-\mu}{2}}\}.\] 
Lemma~\ref{lem:residual}(ii) covers $D$ with one red path; together with
$P_2$, this is a red path cover of size $c=2$, a contradiction.  Hence
$c=2$ is impossible.

It remains to consider the case $c=3$.  From \eqref{eq:small-basic}, we have $w\in\{3,4\}$ and  $r\in\{5,6\}$. First suppose that $w=3$.  Lemma~\ref{lem:basic-bounds} yields $a_0=2$ and $a_1=1$. Let $H$ be the bipartite graph with parts $X$ and $Y$ whose edges are precisely the red edges of $G$ joining $X$ and $Y$. For this graph, the sets on the $Y$-side defined in Lemma~\ref{lem:bip-cover} are precisely the previously defined sets $Y_0$ and $Y_1$, since $d_H(y)=|X|$ if and only if $y$ has no blue neighbour in $X$. Let $z$ be the unique vertex of $Y_1$.  Define $X_0=\{x\in X:d_H(x)=|Y|\}$ and  $X_1=X\setminus X_0$.
We have $|X|=6+r\ge11$ and $\mu\le r-2\le4$, so $|X|>2\mu$.
Furthermore, $|X_1|\le\mu$ and $|X_0|\ge|X|-\mu$. It follows that
\[
   |X_0|>|X_1|=\frac{2|X_1|}{|Y_0|}\,|Y_1|.
\]
Equivalently,
\[
   \frac{|X_0|}{|Y_1|}>\frac{2|X_1|}{|Y_0|}.
\]
All hypotheses of Lemma~\ref{lem:bip-cover} are satisfied.  It yields a red
path cover with at most
\[
   \ceil{\frac{|X|}{|Y|+1}}
   =\ceil{\frac{6+r}{4}}\le3=c
\]
paths, a contradiction.

Finally, suppose that $w=4$.  Then \eqref{eq:small-basic} forces $r=6$, and
$|X|=11$.  Choose $z\in Y_1$ with $|\Nb(z)\cap X|=\mu$ and let $S_z$ be the
corresponding red clique. Let $W=X\setminus S_z$.  Thus
$|W|=10-\mu$ and $1\le\mu\le4$.  Fix $y_0\in Y_0$ and
$y_1\in Y\setminus\{z,y_0\}$. If $\mu=4$, then $|\Nr(z)\cap W|\ge2$.  Choose
$x_1\in\Nr(z)\cap W$ and $x_2\in W\setminus\{x_1\}$; the sequence
$z,x_1,y_0,x_2$ satisfies Lemma~\ref{lem:splice} with $a=2$, a contradiction.
If $\mu=3$, then
\[
|\Nr(z)\cap\Nr(y_1)\cap W|\ge|W|-2\mu\ge1,
\qquad |\Nr(y_1)\cap W|\ge|W|-\mu\ge4.
\]
Because every edge from $y_0$ to $W$ is red, we may greedily choose the three
distinct vertices required by Lemma~\ref{lem:splice} for the ordered
vertices $z,y_1,y_0$.  This again gives a contradiction.

We are left with $\mu\le2$. Every vertex of $Y$ has at least $|X|-\mu\ge9$ red neighbours in $X$. Hence Lemma~\ref{lem:bip-path} yields a red path on $2|Y|=8$ vertices. Since this path lies in the bipartite graph with parts $X$ and $Y$, it contains four vertices from each part; as $|Y|=4$, it contains every vertex of $Y$. Orienting the path from its endpoint in $Y$, write $P_1=z_1x_1z_2x_2z_3x_3z_4x_4$, where $x_1,x_2,x_3,x_4\in X$. Since $|\Nr(z_1)\cap X|\ge9$, we may choose
\[
x_0\in(\Nr(z_1)\cap X)\setminus\{x_1,x_2,x_3,x_4\}.
\]
Consequently, $P_2=x_0z_1x_1z_2x_2z_3x_3z_4x_4$ is a red path. Now let $D=X\setminus\{x_0,x_1,x_2,x_3,x_4\}$ and $h=c-1=2$.  Then $d=6$.  Here \[p=d-h(a_0+1)=4-2a_0.\]
If $p\le0$, Lemma~\ref{lem:residual}(i) applies.  If $p>0$, then $a_0\ge1$
gives $p\le2=h$, while \[p\le2\le\floor{\frac{d-\mu}{2}}.\] Thus Lemma~\ref{lem:residual}(ii) applies.  In either case, $D$ is covered by
$h$ red paths, and these paths together with $P_2$ form a red path cover of
$G$ of size at most $c=3$.  This contradiction shows that $c=3$ is also impossible.
%\end{proof}
\vskip 2mm
By the arguments above, both $c\ge 4$ and $c\le 3$ are impossible, which implies that the minimal counterexample  does not exist, and so Theorem \ref{thm:main} follows.

%\begin{proof}[Proof of Theorem~\ref{thm:main}]
%A minimal counterexample has $c=\floor{\sqrt n}\ge1$.
%Lemma~\ref{lem:c-at-most-three} gives $c\le3$, whereas
%Lemma~\ref{lem:small-c} excludes every value $1\le c\le3$.  Therefore no
%minimal counterexample exists, and the theorem follows.
%\end{proof}

\section*{Declarations}
The authors declare that they have no known competing financial interests or personal relationships that could have appeared to influence the work reported in this paper.
\section*{\bf\Large Availability of Data and Materials}  
Not applicable.
\section*{Acknowledgments}
This research is supported by National Key R\&D Program of China under grant number 2024YFA1013900, NSFC under grant numbers 12471327 and 12401454, Natural Science Foundation of Fujian Province under grant number 2024J01875, Science-Technology Foundation of Putian University under grant number 2023059.


\begin{thebibliography}{99}
\small \setlength{\itemsep}{-.10mm}

\bibitem{Allen2008}
P.~Allen,
Covering two-edge-colored complete graphs with two disjoint monochromatic
cycles,
Combin. Probab. Comput. \textbf{17} (4) (2008)  471--486.

\bibitem{BessyThomasse2010}
S.~Bessy, S.~Thomass\'e,
Partitioning a graph into a cycle and an anticycle, a proof of Lehel's
conjecture,
J. Combin. Theory Ser. B \textbf{100} (2) (2010) 176--180.

\bibitem{ErdosGyarfas1995}
P.~Erd\H{o}s, A.~Gy\'arf\'as,
Vertex covering with monochromatic paths,
Math. Pannon. \textbf{6} (1995)  7--10.

\bibitem{EugsterMousset2018}
M.~Eugster, F.~Mousset,
Vertex covering with monochromatic pieces of few colors,
Electron. J. Combin. \textbf{25} (3) (2018) Paper No. P3.33.


\bibitem{GerencserGyarfas1967}
L.~Gerencs\'er, A.~Gy\'arf\'as,
On Ramsey-type problems,
Ann. Univ. Sci. Budapest. E\"otv\"os Sect. Math. \textbf{10} (1967) 167--170.

\bibitem{Gyarfas1989}
A.~Gy\'arf\'as,
Covering complete graphs by monochromatic paths,
in: Irregularities of Partitions,
Fert\H{o}d, 1986, in: Algorithms Combin. Study Res. Texts, vol. 8, Springer, Berlin, 1989, pp. 89--91.



\bibitem{GyarfasLehel1973}
A.~Gy\'arf\'as, J.~Lehel,
A Ramsey-type problem in directed and bipartite graphs,
Period. Math. Hungar. \textbf{3} (1973) 299--304.

\bibitem{Gyarfas2016}
A.~Gy\'arf\'as,
Vertex covers by monochromatic pieces---a survey of results and problems,
Discrete Math. \textbf{339} (7) (2016) 1970--1977, 7th Cracow Conference on Graph Theory, Rytro 2014. 

\bibitem{LuczakRodlSzemeredi1998}
T.~\L{}uczak, V.~R\"odl, E.~Szemer\'edi,
Partitioning two-colored complete graphs into two monochromatic cycles,
Combin. Probab. Comput. \textbf{7} (1998) 423--436.


\bibitem{PokrovskiyVersteegenWilliams2026}
A.~Pokrovskiy, L.~Versteegen, E.~Williams,
A proof of a conjecture of Erd\H{o}s and Gy\'arf\'as on monochromatic path
covers,
J. Combin. Theory Ser. B \textbf{176} (2026)  551--560.
\end{thebibliography}
\end{document}